\documentclass[10pt]{article}

\usepackage{amsmath,amsbsy,amssymb,latexsym,amsthm,dsfont}
\usepackage[top=2.3cm, bottom=2cm, left=2cm, right=2cm]{geometry}
\usepackage{graphicx}
\usepackage{cite}

\newtheorem{theorem}{Theorem}

\newtheorem{lemma}[theorem]{Lemma}

\newtheorem{remark}[theorem]{Remark}
\newtheorem{example}[theorem]{Example}

\newenvironment{keywords}{\begin{center}
\begin{minipage}[c]{13.4cm} {\bf Keywords:}} {\end{minipage}
\end{center}}

\newenvironment{msc}{\begin{center}
\begin{minipage}[c]{13.4cm} {\bf MSC 2010:}} {\end{minipage}
\end{center}}

\begin{document}

\title{Fractional variational problems\\
depending on indefinite integrals and with delay}

\author{Ricardo Almeida\footnote{ricardo.almeida@ua.pt, CIDMA -- Center for Research and Development in Mathematics and Applications,
Department of Mathematics, University of Aveiro, Portugal.}}

\date{}

\maketitle


\begin{abstract}
The aim of this paper is to exhibit a necessary and sufficient condition of optimality for functionals depending on fractional  integrals and derivatives, on indefinite integrals
and on presence of time delay. We exemplify with one example, where we find analytically the minimizer.
\end{abstract}

\begin{msc}
49K05, 49S05, 26A33, 34A08.
\end{msc}



\section{Introduction}\label{sec:intro}
In this paper we proceed the work started in \cite{Almeida5}, where the authors studied fractional variational problems with the Lagrangian containing not only
 fractional integrals and fractional derivatives, but an indefinite integral as well. With this approach, we tried not only to obtain new results but also
 generalize some already known. The novelty of this paper is that we consider dependence on time delay in the cost functional. Since fractional derivatives are
 characterized by retaining memory, it is natural to state the system at an earlier time and many phenomena have time delays inherent in them. This is a field under
 strong research, namely for optimal control problems, differential equations, biology, etc (see e.g. \cite{Chen,Dehghan,Liu0,Liu,Mo,Udaltsov,Xu,Zhu}).
 For some literature on what this paper concerns, we suggest the reader to \cite{AGRA1,Almeida,Baleanu1,Bhrawy,Chen2,Gastao0,Loghmani,Malinowska,Mozyrska,Yueqiang} for fractional variational problems dealing with Caputo derivative, in \cite{Almeida1} for Lagrangians depending on fractional integrals, and in \cite{Gregory,Nat} when presence of indefinite integrals.
 For a standard variational approach to systems in presence of time delay or more general topics, we suggest the interested reader to the papers \cite{AGRA0,Rosenblueth1,Rosenblueth2,Wang}, and for the fractional
 approach to \cite{Baleanu,Jarad}.

 The paper is organized in the following way. For the reader's convenience, in section \ref{sec:frac} we recall some definitions and results on
 fractional calculus; namely the definitions of fractional integral and fractional derivative, and some fractional integration by parts formulas.
 Section \ref{sec:ELequation} is the main core of the paper: we exhibit a necessary and sufficient condition of optimality for the functional that we purpose to study in  this paper.

\section{Review on fractional calculus}\label{sec:frac}

Let  us  now  explain  the  notation  used. For more, see e.g. \cite{Kilbas,Miller,samko}.

Given a function $f:[a,b]\to\mathbb{R}$, $\alpha\in(0,1)$ and $\beta>0$, the left and right fractional integrals of order $\beta$ of $f$ are respectively
$${_aI_x^\beta}f(x)=\frac{1}{\Gamma(\beta)}\int_a^x (x-t)^{\beta-1}f(t)dt,$$
and
$${_xI_b^\beta}f(x)=\frac{1}{\Gamma(\beta)}\int_x^b(t-x)^{\beta-1} f(t)dt.$$
The left and right Riemann--Liouville fractional derivatives of order $\alpha$ of $f$ are respectively
$${_aD_x^\alpha}f(x)=\frac{1}{\Gamma(1-\alpha)}\frac{d}{dx}\int_a^x(x-t)^{-\alpha}f(t)dt$$
and
$${_xD_b^\alpha}f(x)=\frac{-1}{\Gamma(1-\alpha)}\frac{d}{dx}\int_x^b (t-x)^{-\alpha} f(t)dt.$$
The left and right Caputo fractional derivatives of order $\alpha$ of $f$ are respectively
$${_a^CD_x^\alpha}f(x)=\frac{1}{\Gamma(1-\alpha)}\int_a^x (x-t)^{-\alpha}\frac{d}{dt}f(t)dt$$
and
$${_x^CD_b^\alpha}f(x)=\frac{-1}{\Gamma(1-\alpha)}\int_x^b(t-x)^{-\alpha}\frac{d}{dt} f(t)dt.$$
It is obvious that these operators are linear, and in some sense fractional differentiation and fractional integration are inverse operations.

Caputo fractional derivative seems to be more natural than the Riemann-Liouville fractional derivative. There are two main reasons for that.
The first one is that the Caputo derivative of a constant is zero, while the  Riemann-Liouville  derivative of $f(x)=C$ is $C(x-a)^{-\alpha}/\Gamma(1-\alpha)$. The second one is that the Laplace transform of the Caputo derivative depends on the derivative of integer order of the function
$$(\mathcal{L}\,{_0^CD^\alpha_s}f)(s)=s^\alpha (\mathcal{L}\,f)(s)-\sum_{k=0}^{n-1}s^{\alpha-k-1}\frac{d^kf}{ds^k}(0),$$
in opposite to the Riemann-Liouville  derivative  that uses fractional integrals evaluated at the initial value.

A basic result needed to apply variational methods is the integration by parts formula, that in case for fractional integrals is
\begin{equation}\label{Int2}\displaystyle\int_{a}^{b}  g(x) \cdot {_aI_x^\beta}f(x)dx=\int_a^b f(x) \cdot {_x I_b^\beta} g(x)dx \, ,\end{equation}
and for Caputo fractional derivatives, we have
\begin{equation}\label{Int}\int_{a}^{b}g(x)\cdot {_a^C D_x^\alpha}f(x)dx=\int_a^b f(x)\cdot {_x D_b^\alpha} g(x)dx+\left[f(x)\cdot{_xI_b^{1-\alpha}}g(x)\right]_a^b.
\end{equation}

Formula \eqref{Int} can be generalized in a way to include the case where the lower bound of the integral is distinct of the lower bound of the Caputo derivative.

\begin{lemma}\label{LemmaInt} Let $f$ and $g$ be two functions of class $C^1$ on $[a,b]$, and let $r\in(a,b)$. Then
\begin{multline}\label{GenInt}\int_{r}^{b}g(x)\cdot {_a^C D_x^\alpha}f(x)dx=\int_r^b f(x)\cdot {_x D_b^\alpha} g(x)dx\\
-\int_a^r\frac{f(x)}{\Gamma(1-\alpha)}\, \frac{d}{dx}\left(\int_r^b (t-x)^{-\alpha}g(t)\,dt\right)dx-
\frac{f(a)}{\Gamma(1-\alpha)}\int_r^b(t-a)^{-\alpha}g(t)dt.\end{multline}
\end{lemma}
\begin{proof} It follows due the next relations:
$$\begin{array}{ll}
\displaystyle \int_{r}^{b}g(x)\cdot {_a^C D_x^\alpha}f(x)dx & =\displaystyle \int_{a}^{b}g(x)\cdot {_a^C D_x^\alpha}f(x)dx-\int_{a}^{r}g(x)\cdot {_a^C D_x^\alpha}f(x)dx\\
                    & = \displaystyle\int_{a}^{b}f(x)\cdot {_x D_b^\alpha}g(x)dx+ \left[f(x) \cdot{_xI_b^{1-\alpha}}g(x) \right]_a^b\\
                    &\quad - \displaystyle \int_{a}^{r}f(x)\cdot {_x D_r^\alpha}g(x)dx-\left[f(x)\cdot {_xI_r^{1-\alpha}}g(x) \right]_a^r\\
                     &=\displaystyle\int_{r}^{b}f(x)\cdot {_x D_b^\alpha}g(x)dx+\int_{a}^{r}f(x)\cdot \left({_x D_b^\alpha}g(x)-{_x D_r^\alpha}g(x)\right)dx\\
                     & \displaystyle\quad +\left[f(x)\cdot {_xI_b^{1-\alpha}}g(x) \right]_a^b-\left[f(x)\cdot {_xI_r^{1-\alpha}}g(x) \right]_a^r\\
                     &=\displaystyle\int_r^b f(x)\cdot {_x D_b^\alpha} g(x)dx\\
&\displaystyle-\int_a^r\frac{f(x)}{\Gamma(1-\alpha)}\, \frac{d}{dx}\left(\int_r^b (t-x)^{-\alpha}g(t)\,dt\right)dx-
\frac{f(a)}{\Gamma(1-\alpha)}\int_r^b(t-a)^{-\alpha}g(t)dt.
\end{array}$$
\end{proof}


\section{The Euler-Lagrange equation}\label{sec:ELequation}
\label{sec:ELequation}

The cost functional that we will study is given by the expression
\begin{equation}
\label{funct}
J(y)=\int_a^b L(x,y(x),{^C_aD_x^\alpha}y(x),{_aI_x^\beta}y(x),z(x), y(x-\tau), y'(x-\tau))dx,
\end{equation}
defined on $C^1[a-\tau,b]$, where
$$\left\{
\begin{array}{l}
\tau>0, \mbox{ and } \tau<b-a,\\
\alpha\in(0,1) \mbox{ and }\beta>0,\\
z(x)=\int_a^x l(t,y(t),{^C_aD_t^\alpha}y(t),{_aI_t^\beta}y(t))dt,\\
L=L(x,y,v,w,z,y_\tau,v_\tau) \mbox{ and } l=l(x,y,v,w) \mbox{ are of class } C^1
\end{array}\right.$$
and the admissible functions are such that
$$\left\{
\begin{array}{l}
{^C_aD_x^\alpha}y(x) \mbox{ and } {_aI_x^\beta}y(x) \mbox{ exist and are continuous on } [a,b],\\
y(b)=y_b\in \mathbb R,\\
y(x)=\phi(x), \mbox{ for all } x\in [a-\tau,a], \, \phi \mbox{ a fixed function.}
\end{array}\right.$$
The set of variation functions of $y$ that we will consider are those of type $y+\epsilon h$, such that $|\epsilon| \ll1$ and $h\in C^1[a-\tau,b]$ with
$$\left\{
\begin{array}{l}
h(b)=0,\\
h(x)=0, \mbox{ for all } x\in [a-\tau,a].
\end{array}\right.$$

An important result in variational calculus is the so called du Bois-Reymond Theorem:

\begin{theorem}\label{dubois} (see e.g. \cite{Brunt}) Let $f:[a,b]\to\mathbb R$ be a continuous functions, and suppose that the relation
$$\int_a^b f(x)h(x)dx=0$$
holds for every $h\in C^k[a,b]$, with $k\geq 0$. Then $f(x)=0$ on $[a,b]$.
\end{theorem}

Theorem \ref{dubois} still holds if we impose the auxiliary conditions $h(a)=h(b)=0$.

From now on, to simplify writing, by $[y](x)$ and $\{y\}(x)$ we denote the vectors
$$[y](x)=(x,y(x),{^C_aD_x^\alpha}y(x),{_aI_x^\beta}y(x),z(x), y(x-\tau), y'(x-\tau))\quad \mbox{and}\quad \{y\}(x)=(x,y(x),{^C_aD_x^\alpha}y(x),{_aI_x^\beta}y(x)).$$
Let $y$ be a minimizer or maximizer of $J$ as in \eqref{funct}. As it is known, at the extremizers of the functional we have
$$\frac{d}{d\epsilon}J(y+\epsilon h)=0,$$
where $y+\epsilon h$ is any variation of $y$. Proceeding with the necessary calculations, we deduce that
\begin{multline*}
\int_a^b \left[ \frac{\partial L}{\partial y}[y](x)h(x)
+ \frac{\partial L}{\partial v}[y](x){^C_aD^\alpha_x}h(x)
+ \frac{\partial L}{\partial w}[y](x){_aI^\beta_x}h(x)\right.\\
\left.+\frac{\partial L}{\partial z}[y](x)\int_a^x\left(
\frac{\partial l}{\partial y}\{y\}(t)h(t)
+\frac{\partial l}{\partial v}\{y\}(t){^C_aD^\alpha_t}h(t)
+\frac{\partial l}{\partial w}\{y\}(t){_aI^\beta_t}h(t)\right)dt\right.\\
\left. +\frac{\partial L}{\partial y_\tau}[y](x)h(x-\tau)
+ \frac{\partial L}{\partial v_\tau}[y](x)h'(x-\tau)\right]dx=0.
\end{multline*}
Next, we use the following relations

\textit{R1:}
$$\int_a^b  \frac{\partial L}{\partial y}[y](x)h(x)dx=
\int_a^{b-\tau} \frac{\partial L}{\partial y}[y](x)h(x)dx+\int_{b-\tau}^b \frac{\partial L}{\partial y}[y](x)h(x)dx;$$

\textit{R2:} Since $h(a)=0$, and using formulas \eqref{Int} and \eqref{GenInt}
\begin{align*}
 \int_a^b \frac{\partial L}{\partial v}[y](x){^C_aD^\alpha_x}h(x) dx&=\int_a^{b-\tau}
                   \frac{\partial L}{\partial v}[y](x){^C_aD^\alpha_x}h(x) dx
       +\int_{b-\tau}^b \frac{\partial L}{\partial v}[y](x){^C_aD^\alpha_x}h(x) dx\\
     &=\int_a^{b-\tau} {_x D_{b-\tau}^\alpha} \left(\frac{\partial L}{\partial v}[y](x) \right)h(x)dx
     +\left[{_x I_{b-\tau}^{1-\alpha}} \left(\frac{\partial L}{\partial v}[y](x) \right)h(x)\right]_a^{b-\tau}\\
     &\quad +\int_{b-\tau}^b {_x D_{b}^\alpha} \left(\frac{\partial L}{\partial v}[y](x) \right)h(x)dx
     -\frac{h(a)}{\Gamma(1-\alpha)}\int_{b-\tau}^b(t-a)^{-\alpha}\frac{\partial L}{\partial v}[y](t)dt  \\
     &\quad -\int_a^{b-\tau}\frac{h(x)}{\Gamma(1-\alpha)}\, \frac{d}{dx}
     \left( \int_{b-\tau}^b(t-x)^{-\alpha} \frac{\partial L}{\partial v}[y](t)dt\right) dx\\
     &=\int_a^{b-\tau} \left[{_x D_{b-\tau}^\alpha} \left(\frac{\partial L}{\partial v}[y](x)\right)-
     \frac{1}{\Gamma(1-\alpha)}\, \frac{d}{dx}\left( \int_{b-\tau}^b(t-x)^{-\alpha} \frac{\partial L}{\partial v}[y](t)dt\right)\right]h(x)dx\\
     &\quad + \int_{b-\tau}^b {_x D_{b}^\alpha} \left(\frac{\partial L}{\partial v}[y](x) \right)h(x)dx;\\
\end{align*}

\textit{R3:} Using equation \eqref{Int2} and Lemma 2(b) of \cite{Baleanu},
\begin{align*}\int_a^b \frac{\partial L}{\partial w}[y](x){_aI^\beta_x}h(x) dx&=
  \int_a^{b-\tau} \frac{\partial L}{\partial w}[y](x){_aI^\beta_x}h(x) dx+\int_{b-\tau}^b \frac{\partial L}{\partial w}[y](x){_aI^\beta_x}h(x) dx\\
&= \int_a^{b-\tau} {_xI^\beta_{b-\tau}}\left(\frac{\partial L}{\partial w}[y](x)\right)h(x) dx\\
  &\quad +\int_{b-\tau}^b {_xI^\beta_b}\left(\frac{\partial L}{\partial w}[y](x)\right)h(x) dx
    +\frac{1}{\Gamma(\beta)}\int_a^{b-\tau}h(x)\left( \int_{b-\tau}^b (t-x)^{\beta-1} \frac{\partial L}{\partial w}[y](t)dt \right)dx\\
  &= \int_a^{b-\tau} \left[ {_xI^\beta_{b-\tau}}\left(\frac{\partial L}{\partial w}[y](x)\right) +
     \frac{1}{\Gamma(\beta)}  \left( \int_{b-\tau}^b (t-x)^{\beta-1} \frac{\partial L}{\partial w}[y](t)dt \right) \right]h(x) dx\\
     &\quad +\int_{b-\tau}^b {_xI^\beta_b}\left(\frac{\partial L}{\partial w}[y](x)\right)h(x) dx;
\end{align*}

\textit{R4:} Using standard integration by parts,
\begin{align*}
\int_a^b & \frac{\partial L}{\partial z}[y](x)\left(\int_a^x\frac{\partial l}{\partial y}\{y\}(t)h(t) dt \right) dx
=\int_a^b \left( -\frac{d}{dx}\int_x^b\frac{\partial L}{\partial z}[y](t)dt \right)
\left( \int_a^x \frac{\partial l}{\partial y}\{y\}(t)h(t) dt \right) dx\\
&= \int_a^b \left(\int_x^b\frac{\partial L}{\partial z}[y](t)dt \right)
\frac{\partial l}{\partial y}\{y\}(x)h(x) \, dx\\
&= \int_a^{b-\tau} \left(\int_x^b\frac{\partial L}{\partial z}[y](t)dt \right)\frac{\partial l}{\partial y}\{y\}(x)h(x) dx
+ \int_{b-\tau}^b \left(\int_x^b\frac{\partial L}{\partial z}[y](t)dt \right)\frac{\partial l}{\partial y}\{y\}(x)h(x) dx;
\end{align*}

\textit{R5:} Using standard integration by parts, formulas \eqref{Int} and \eqref{GenInt} and since $h(a)=0$,
\begin{align*}
&\int_a^b  \frac{\partial L}{\partial z}[y](x)\left(\int_a^x\frac{\partial l}{\partial v}\{y\}(t){^C_aD^\alpha_t}h(t) dt \right) dx
= \int_a^b \left( -\frac{d}{dx}\int_x^b\frac{\partial L}{\partial z}[y](t)dt \right)\left(\int_a^x \frac{\partial l}{\partial v}\{y\}(t){^C_aD^\alpha_t}h(t)dt\right)dx\\
&= \int_a^b \left(\int_x^b\frac{\partial L}{\partial z}[y](t)dt \right)\frac{\partial l}{\partial v}\{y\}(x){^C_aD^\alpha_x}h(x) dx\\
&= \int_a^{b-\tau} \left(\int_x^b\frac{\partial L}{\partial z}[y](t)dt \right)\frac{\partial l}{\partial v}\{y\}(x){^C_aD^\alpha_x}h(x)dx
  +\int_{b-\tau}^b \left(\int_x^b\frac{\partial L}{\partial z}[y](t)dt \right)\frac{\partial l}{\partial v}\{y\}(x){^C_aD^\alpha_x}h(x)dx\\
&= \int_a^{b-\tau} {_xD^\alpha_{b-\tau}}\left(\int_x^b\frac{\partial L}{\partial z}[y](t)dt\frac{\partial l}{\partial v}\{y\}(x)\right)h(x)dx
+\left[ {_xI^{1-\alpha}_{b-\tau}}\left(\int_x^b\frac{\partial L}{\partial z}[y](t)dt\frac{\partial l}{\partial v}\{y\}(x) \right) h(x) \right]_a^{b-\tau}\\
  & \quad +\int_{b-\tau}^b {_xD^\alpha_b}\left(\int_x^b\frac{\partial L}{\partial z}[y](t)dt\frac{\partial l}{\partial v}\{y\}(x)\right)h(x)dx
  -\frac{h(a)}{\Gamma(1-\alpha)}\int_{b-\tau}^b(t-a)^{-\alpha}\int_t^b \frac{\partial L}{\partial z}[y](k)dk \frac{\partial l}{\partial v}\{y\}(t)dt  \\
     &\quad -\int_a^{b-\tau}\frac{h(x)}{\Gamma(1-\alpha)}\, \frac{d}{dx}
     \left( \int_{b-\tau}^b(t-x)^{-\alpha}\int_t^b \frac{\partial L}{\partial z}[y](k)dk \frac{\partial l}{\partial v}\{y\}(t)dt \right)dx\\
     &=\int_a^{b-\tau} \left[{_xD^\alpha_{b-\tau}}\left(\int_x^b\frac{\partial L}{\partial z}[y](t)dt\frac{\partial l}{\partial v}\{y\}(x)\right)
     -\frac{1}{\Gamma(1-\alpha)} \frac{d}{dx}
     \left( \int_{b-\tau}^b(t-x)^{-\alpha}\int_t^b \frac{\partial L}{\partial z}[y](k)dk \frac{\partial l}{\partial v}\{y\}(t)dt \right)\right]h(x)dx\\
     &\quad + \int_{b-\tau}^b {_xD^\alpha_b}\left(\int_x^b\frac{\partial L}{\partial z}[y](t)dt\frac{\partial l}{\partial v}\{y\}(x)\right)h(x)dx
\end{align*}

\textit{R6:} Using standard integration by parts, equation \eqref{Int2} and Lemma 2(b) of \cite{Baleanu},
\begin{align*}
&\int_a^b \frac{\partial L}{\partial z}[y](x)\left(\int_a^x\frac{\partial l}{\partial w}\{y\}(t){_aI^\beta_t}h(t) dt \right) dx=
\int_a^b \left( -\frac{d}{dx}\int_x^b\frac{\partial L}{\partial z}[y](t)dt\right)\left(\int_a^x\frac{\partial l}{\partial w}\{y\}(t){_aI^\beta_t}h(t) dt \right) dx\\
&=\int_a^b \left( \int_x^b\frac{\partial L}{\partial z}[y](t)dt\right) \frac{\partial l}{\partial w}\{y\}(x){_aI^\beta_x}h(x) dx\\
&=\int_a^{b-\tau}  \left(\int_x^b\frac{\partial L}{\partial z}[y](t)dt\right) \frac{\partial l}{\partial w}\{y\}(x){_aI^\beta_x}h(x) dx
+\int_{b-\tau}^b  \left(\int_x^b\frac{\partial L}{\partial z}[y](t)dt\right) \frac{\partial l}{\partial w}\{y\}(x){_aI^\beta_x}h(x) dx\\
&=\int_a^{b-\tau}\left[ {_xI^\beta_{b-\tau}}\left(\int_x^b\frac{\partial L}{\partial z}[y](t)dt\frac{\partial l}{\partial w}\{y\}(x)\right)
+\frac{1}{\Gamma(\beta)}\left(  \int_{b-\tau}^b (t-x)^{\beta-1}\int_t^b  \frac{\partial L}{\partial z}[y](k)dk \, \frac{\partial l}{\partial w}\{y\}(t)dt\right)
\right]h(x) dx \\
&+ \int_{b-\tau}^b {_xI^\beta_b}\left(\int_x^b\frac{\partial L}{\partial z}[y](t)dt\frac{\partial l}{\partial w}\{y\}(x)\right)h(x)dx.
\end{align*}

\textit{R7:} Since $h(x)=0$ for all $x\in[a-\tau,a]$,
$$\int_a^b \frac{\partial L}{\partial y_\tau}[y](x)h(x-\tau)dx=\int_{a-\tau}^{b-\tau} \frac{\partial L}{\partial y_\tau}[y](x+\tau)h(x)dx
=\int_a^{b-\tau} \frac{\partial L}{\partial y_\tau}[y](x+\tau)h(x)dx$$

\textit{R8:} Since $h(x)=0$ for all $x\in[a-\tau,a]$, using standard integration by parts, we have
$$\int_a^b \frac{\partial L}{\partial v_\tau}[y](x)h'(x-\tau)dx=\int_a^{b-\tau} \frac{\partial L}{\partial v_\tau}[y](x+\tau)h'(x)dx=
\frac{\partial L}{\partial v_\tau}[y](b)h(b-\tau)-\int_a^{b-\tau} \frac{d}{dx}\left(\frac{\partial L}{\partial v_\tau}[y](x+\tau)\right)h(x)dx$$

We are now in position to obtain a necessary condition of optimality when in presence of the time delay $\tau>0$.

\begin{theorem}\label{Teo1} If $y$ is a minimizer or maximizer of $J$ as in \eqref{funct}, then $y$ is a solution of the system of equations
\begin{enumerate}
\item $\displaystyle \frac{\partial L}{\partial v_\tau}[y](b)=0$;
\item for every $x\in[a,b-\tau]$,
\begin{align*}
&\frac{\partial L}{\partial y}[y](x)+{_xD^\alpha_{b-\tau}}\left( \frac{\partial L}{\partial v}[y](x) \right)
- \frac{1}{\Gamma(1-\alpha)}\, \frac{d}{dx}\left( \int_{b-\tau}^b(t-x)^{-\alpha} \frac{\partial L}{\partial v}[y](t)dt\right)\\
&\quad +{_xI_{b-\tau}^\beta}\left(\frac{\partial L}{\partial w}[y](x)\right)
+\frac{1}{\Gamma(\beta)}  \left( \int_{b-\tau}^b (t-x)^{\beta-1} \frac{\partial L}{\partial w}[y](t)dt \right)
 +\int_x^b \frac{\partial L}{\partial z}[y](t)dt  \frac{\partial l}{\partial y}\{y\}(x)\\
&\quad+{_xD^\alpha_{b-\tau}}\left( \int_x^b \frac{\partial L}{\partial z}[y](t)dt \frac{\partial l}{\partial v}\{y\}(x)\right)
-\frac{1}{\Gamma(1-\alpha)} \frac{d}{dx}
\left( \int_{b-\tau}^b(t-x)^{-\alpha}\int_t^b \frac{\partial L}{\partial z}[y](k)dk \frac{\partial l}{\partial v}\{y\}(t)dt \right)\\
&\quad +{_xI^\beta_{b-\tau}}\left( \int_x^b \frac{\partial L}{\partial z}[y](t)dt \frac{\partial l}{\partial w}\{y\}(x)  \right)
+\frac{1}{\Gamma(\beta)}\left(  \int_{b-\tau}^b (t-x)^{\beta-1}\int_t^b  \frac{\partial L}{\partial z}[y](k)dk \, \frac{\partial l}{\partial w}\{y\}(t)dt\right)\\
& \quad +\frac{\partial L}{\partial y_\tau}[y](x+\tau) -\frac{d}{dx}\frac{\partial L}{\partial v_\tau}[y](x+\tau)=0;
\end{align*}
\item for every $x\in[b-\tau,b]$,
\begin{align*}
&\frac{\partial L}{\partial y}[y](x)+{_xD^\alpha_b}\left( \frac{\partial L}{\partial v}[y](x) \right)+{_xI_b^\beta}\left(\frac{\partial L}{\partial w}[y](x)\right)\\
&\quad +\int_x^b \frac{\partial L}{\partial z}[y](t)dt \frac{\partial l}{\partial y}\{y\}(x)
+{_xD^\alpha_b}\left( \int_x^b \frac{\partial L}{\partial z}[y](t)dt \frac{\partial l}{\partial v}\{y\}(x)\right)
+{_xI^\beta_b}\left( \int_x^b \frac{\partial L}{\partial z}[y](t)dt \frac{\partial l}{\partial w}\{y\}(x)  \right)=0.
\end{align*}
\end{enumerate}
\end{theorem}

\begin{proof} If follows combining relations \textit{R1}-\textit{R8}, the arbitrariness of $h$ and from Theorem \ref{dubois}.
\end{proof}

\begin{example}\label{example2}
Consider the function
$$y_\alpha(x)=\left\{
\begin{array}{lll}
0&\mbox{ if }& x\in[-1,0]\\
x^{\alpha+1}&\mbox{ if }& x\in[0,2].\\
\end{array}\right.$$
Then
$${^C_0D_x^\alpha}y_\alpha(x)=\Gamma(\alpha+2)x.$$
For the cost functional, let
\begin{equation}
\label{example}
J(y)=\int_0^2 ({^C_0D_x^\alpha}y(x)-\Gamma(\alpha+2)x)^2+z(x)+(y'(x-1)-y'_\alpha(x-1))^2dx,
\end{equation}
where
$$z(x)=\int_0^x (y(t)-t^{\alpha+1})^2 \, dt,$$
defined on the set $C^1[-1,2]$, under the constraints
$$\left\{
\begin{array}{l}
y(2)=2^{\alpha+1},\\
y(x)=0, \mbox{ for all } x\in [-1,0].
\end{array}\right.$$
Since $J(y)\geq0$ for all admissible functions $y$, and $J(y_\alpha)=0$, we have that $y_\alpha$ is a minimizer of $J$ and zero is its minimum value.
Equations \textit{1-3} of Theorem \ref{Teo1} applied to $J$ read as
\begin{enumerate}
\item $\displaystyle \left[y'(x-1)-y'_\alpha(x-1)\right]_{x=2}=0$;
\item for every $x\in[0,1]$,
\begin{align*}
&{_xD_1^\alpha}({^C_0D_x^\alpha}y(x)-\Gamma(\alpha+2)x)
- \frac{1}{\Gamma(1-\alpha)}\,\frac{d}{dx}\left( \int_{1}^2(t-x)^{-\alpha}({^C_0D_t^\alpha}y(t)-\Gamma(\alpha+2)t)dt\right)\\
&\quad +\int_x^21dt \, (y(x)-x^{\alpha+1})-\frac{d}{dx}\left(y'(x)-y'_\alpha(x)\right)=0\\
\end{align*}
\item for every $x\in[1,2]$,
\begin{align*}
&{_xD_2^\alpha}({^C_0D_x^\alpha}y(x)-\Gamma(\alpha+2)x)+\int_x^21dt \, (y(x)-x^{\alpha+1})=0.\\
\end{align*}
\end{enumerate}

Obviously, $y_\alpha$ is a solution for the three previous conditions 1--3.
\end{example}

\begin{remark} In \cite{Jarad} fractional variational problems in presence of Caputo derivatives and delays are considered. Since the variational functions $h$ are
chosen in such a way that take the value zero at the extrema, the Caputo and the Riemann-Liouville derivative of these functions  are equal.
Using a general integration by parts formula of \cite{Baleanu} similar to our Lemma \ref{LemmaInt}, but for Riemann-Liouville derivative, the problem of
\cite{Jarad} is solved for Caputo derivative. Here we choose to obtain the equivalent formula of  \cite{Baleanu} for the Caputo derivative.
\end{remark}

\begin{remark} Consider the case when $\alpha$ goes to 1 and $\beta$ goes to zero. If so, we obtain the standard functional derived from \eqref{funct}:
\begin{equation}
\label{funct2}
J(y)=\int_a^b L(x,y(x),y'(x),z(x), y(x-\tau), y'(x-\tau))dx,
\end{equation}
where
$$z(x)=\int_a^x l(t,y(t),y'(t))dt,$$
defined for $y\in C^1[a-\tau,b]$ satisfying the boundary conditions
$$\left\{
\begin{array}{l}
y(b)=y_b\in \mathbb R,\\
y(x)=\phi(x), \mbox{ for all } x\in [a-\tau,a].
\end{array}\right.$$
If $y$ is a minimizer or maximizer of $J$ as in \eqref{funct2}, then $y$ is a solution of the system of equations
\begin{enumerate}
\item $\displaystyle \frac{\partial L}{\partial v_\tau}[y](b)=0$;
\item for every $x\in[a,b-\tau]$,
\begin{align*}
&\frac{\partial L}{\partial y}[y](x)-\frac{d}{dx}\left( \frac{\partial L}{\partial v}[y](x) \right)
+\int_x^b \frac{\partial L}{\partial z}[y](t)dt  \frac{\partial l}{\partial y}\{y\}(x)\\
&\quad -\frac{d}{dx}\left(\int_x^b \frac{\partial L}{\partial z}[y](t)dt \frac{\partial l}{\partial v}\{y\}(x)\right)
+\frac{\partial L}{\partial y_\tau}[y](x+\tau) -\frac{d}{dx}\frac{\partial L}{\partial v_\tau}[y](x+\tau)=0;
\end{align*}
\item for every $x\in[b-\tau,b]$,
$$\frac{\partial L}{\partial y}[y](x)-\frac{d}{dx}\left(\frac{\partial L}{\partial v}[y](x) \right)
+\int_x^b \frac{\partial L}{\partial z}[y](t)dt \frac{\partial l}{\partial y}\{y\}(x)
-\frac{d}{dx}\left( \int_x^b \frac{\partial L}{\partial z}[y](t)dt \frac{\partial l}{\partial v}\{y\}(x)\right)=0.$$
\end{enumerate}
This result is apparently new also.
\end{remark}

\begin{remark} Theorem \ref{Teo1} can be generalized for functionals with several dependent variables. Let us consider
\begin{equation}
\label{funct3}
J(y_1,\ldots,y_n)=\int_a^b L(x,y_1(x),\ldots,y_n(x),{^C_aD_x^{\alpha_1}}y_1(x),\ldots,{^C_aD_x^{\alpha_n}}y_n(x),\end{equation}
$${_aI_x^{\beta_1}}y_1(x),\ldots,{_aI_x^{\beta_n}}y_n(x),z(x), y_1(x-\tau_1),\ldots,y_n(x-\tau_n), y_1'(x-\tau_1),\ldots, y_n'(x-\tau_n))dx,$$
defined on $C^1 \prod^{n}_{i=1} [a-\tau_i,b]$, where for all $i\in\{1,\ldots,n\}$,
$$\left\{
\begin{array}{l}
\tau_i>0, \mbox{ and } \tau_i<b-a,\\
\alpha_i\in(0,1) \mbox{ and }\beta_i>0,\\
z(x)=\int_a^x l(t,y_1(t),\ldots,y_n(t),{^C_aD_t^{\alpha_1}}y_1(t),\ldots,{^C_aD_t^{\alpha_n}}y_n(t),{_aI_t^{\beta_1}}y_1(t),\ldots,{_aI_t^{\beta_n}}y_n(t))dt,\\
L=L(x,y_1,\ldots,y_n,v_1,\ldots,v_n,w_1,\ldots,w_n,z,y_{\tau_1},\ldots,y_{\tau_n},v_{\tau_1},\ldots,v_{\tau_n}) \mbox{ and } \\
l=l(x,y_1,\ldots,y_n,v_1,\ldots,v_n,w_1,\ldots,w_n) \mbox{ are of class } C^1
\end{array}\right.$$
and the admissible functions are such that
$$\left\{
\begin{array}{l}
{^C_aD_x^{\alpha_i}}y_i(x) \mbox{ and } {_aI_x^{\beta_i}}y_i(x) \mbox{ exist and are continuous on } [a,b],\\
y_i(b)=y_{b_i}\in \mathbb R,\\
y_i(x)=\phi_i(x), \mbox{ for all } x\in [a-\tau_i,a], \, \phi_i \mbox{ a fixed function.}
\end{array}\right.$$
If the \textit{n}-uple function $(y_1,\ldots,y_n)$ is a minimizer or maximizer of $J$ as in \eqref{funct3},
then $(y_1,\ldots,y_n)$ is a solution of the system of equations similar to the ones of \textit{1-3} of Theorem \ref{Teo1}, replacing the variables
$$y\to y_i,\quad v\to v_i,\quad w\to w_i,\quad y_\tau\to y_{\tau_i},\quad v_\tau\to v_{\tau_i},\quad \alpha\to\alpha_i,\quad \beta\to\beta_i,\quad\tau\to\tau_i,$$
for all $i\in\{1,\dots,n\}$.
\end{remark}


\section{Sufficient condition}\label{sec:SufConditions}

Assuming some convexity conditions on the Lagrangian $L$ and on the supplementary function $l$, we may present a necessary condition that guarantees the existence
of minimizers for the problem. For convenience, recall the definition of convex and concave function. Given $k\in\{1,\ldots,n\}$ and $f:D\subseteq\mathbb{R}^n\to \mathbb{R}$ a
 function differentiable with respect to $x_k,\ldots,x_n$, we say that $f$ is convex (resp. concave) in  $(x_k,\ldots,x_n)$ if
$$f(x_1+c_1,\ldots,x_n+c_n)-f(x_1,\ldots,x_n)\geq \, (resp. \leq) \, \sum_{i=k}^n\frac{\partial f}{\partial x_i}(x_1,\ldots,x_n)c_i,$$
for all $(x_1,\ldots,x_n),(x_1+c_1,\ldots,x_n+c_n)\in D$.

\begin{theorem} Let $y$ be a function satisfying conditions 1--3 of Theorem \ref{Teo1}. If $L$ is convex in $(y,v,w,z,y_\tau,v_\tau)$ and one of the two following
conditions are met
\begin{enumerate}
\item $l$ is convex in $(y,v,w)$ and $\frac{\partial L}{\partial z}[y](x) \geq 0$ for all $x \in [a,b]$,
\item $l$ is concave in $(y,v,w)$ and $\frac{\partial L}{\partial z}[y](x) \leq 0$ for all $x \in [a,b]$,
\end{enumerate}
then $y$ is a minimizer of the functional $J$ as in \eqref{funct}.
\end{theorem}

\begin{proof}
Let $y+h$ be a variation of $y$. Using relations  \textit{R1}-\textit{R8} of Section \ref{sec:ELequation}, we have that
$$\begin{array}{ll}
J(y+h) - J(y) & \displaystyle\geq \int_a^{b-\tau}\left[  \frac{\partial L}{\partial y}[y](x)+{_xD^\alpha_{b-\tau}}\left( \frac{\partial L}{\partial v}[y](x) \right)
- \frac{1}{\Gamma(1-\alpha)}\, \frac{d}{dx}\left( \int_{b-\tau}^b(t-x)^{-\alpha} \frac{\partial L}{\partial v}[y](t)dt\right)\right.\\
&\displaystyle\quad+{_xI_{b-\tau}^\beta}\left(\frac{\partial L}{\partial w}[y](x)\right)+\frac{1}{\Gamma(\beta)}  \left( \int_{b-\tau}^b (t-x)^{\beta-1} \frac{\partial L}{\partial w}[y](t)dt \right)\\
&\displaystyle\quad+\int_x^b \frac{\partial L}{\partial z}[y](t)dt  \frac{\partial l}{\partial y}\{y\}(x)
+{_xD^\alpha_{b-\tau}}\left( \int_x^b \frac{\partial L}{\partial z}[y](t)dt \frac{\partial l}{\partial v}\{y\}(x)\right)\\
&\displaystyle\quad-\frac{1}{\Gamma(1-\alpha)} \frac{d}{dx}
\left( \int_{b-\tau}^b(t-x)^{-\alpha}\int_t^b \frac{\partial L}{\partial z}[y](k)dk \frac{\partial l}{\partial v}\{y\}(t)dt \right)\\
 &\displaystyle\quad+{_xI^\beta_{b-\tau}}\left( \int_x^b \frac{\partial L}{\partial z}[y](t)dt \frac{\partial l}{\partial w}\{y\}(x)  \right)
+\frac{1}{\Gamma(\beta)}\left(  \int_{b-\tau}^b (t-x)^{\beta-1}\int_t^b  \frac{\partial L}{\partial z}[y](k)dk \, \frac{\partial l}{\partial w}\{y\}(t)dt\right)\\
&\displaystyle\left.\quad +\frac{\partial L}{\partial y_\tau}[y](x+\tau) -\frac{d}{dx}\frac{\partial L}{\partial v_\tau}[y](x+\tau)\right] h(x)dx\\
&\displaystyle\quad+ \int_{b-\tau}^b\left[  \frac{\partial L}{\partial y}[y](x)+{_xD^\alpha_b}\left( \frac{\partial L}{\partial v}[y](x) \right)+{_xI_b^\beta}\left(\frac{\partial L}{\partial w}[y](x)\right)+\int_x^b \frac{\partial L}{\partial z}[y](t)dt \frac{\partial l}{\partial y}\{y\}(x)\right.\\
&\displaystyle\quad\left. +{_xD^\alpha_b}\left( \int_x^b \frac{\partial L}{\partial z}[y](t)dt \frac{\partial l}{\partial v}\{y\}(x)\right)
+{_xI^\beta_b}\left( \int_x^b \frac{\partial L}{\partial z}[y](t)dt \frac{\partial l}{\partial w}\{y\}(x)  \right)\right] h(x)dx\\
&\displaystyle\quad+ \frac{\partial L}{\partial v_\tau}[y](b)h(b-\tau)=0.\\
\end{array}$$
\end{proof}

For example, report to the example \ref{example2}. For this case,
$$L(x,y,v,w,z,y_\tau,v_\tau) =(v-\Gamma(\alpha+2)x)^2+z+(y_\tau-(\alpha+1)(x-1)^\alpha)^2\mbox{ and }l(x,y,v,w)=(y-x^{\alpha+1})^2 $$
are both convex, and $\frac{\partial L}{\partial z}[y](x) =1$. Observe that $y_\alpha$ is a solution of equations  \textit{1--3} of Theorem \ref{Teo1}, and in fact is a
minimizer of $J$.

\section{Conclusion}

The aim of the paper is to generalize the main result of \cite{Almeida5}, by considering delays in our system. Necessary conditions are proven in case the Lagrange function depends on fractional derivatives and on indefinite integral as well. For future work, we will study numerical tools to solve directly these kind of problems, avoiding to solve analytically fractional differential equations. 
 

\section*{Acknowledgements}

This work was supported by Portuguese funds through the CIDMA - Center for Research and Development in Mathematics and Applications, and the Portuguese Foundation for Science and Technology ("FCT–-Funda\c{c}\~{a}o para a Ci\^{e}ncia e a Tecnologia"), within project PEst-OE/MAT/UI4106/2014.


\end{document}